\documentclass[11pt,reqno]{amsart}
\usepackage[utf8]{inputenc}
\usepackage[T1]{fontenc}
\usepackage{amsmath,amssymb,amsthm}
\usepackage[margin=3cm]{geometry}
\usepackage{enumitem}
\usepackage[protrusion=true,expansion=false]{microtype}

\theoremstyle{plain}
\newtheorem{theorem}{Theorem}[section]
\newtheorem{proposition}[theorem]{Proposition}
\newtheorem{lemma}[theorem]{Lemma}
\newtheorem{corollary}[theorem]{Corollary}
\theoremstyle{definition}
\newtheorem{definition}[theorem]{Definition}

\theoremstyle{remark}

\DeclareMathOperator{\ex}{exp}
\newcommand{\Sred}{\mathsf S_{\mathrm{red}}}
\newcommand{\Lamone}{\Lambda_{\mathbf1}}

\begin{document}

\title[The reduced Smith group of a Kneser graph]
{The reduced Smith group of a Kneser graph and its application to group-valued magic maps}

\author{Ahmet Batal}
\address{Department of Mathematics, Izmir Institute of Technology, 35430, Urla, Izmir, Turkey}
\email{ahmetbatal@iyte.edu.tr}


\begin{abstract}
For a finite abelian group $\Gamma$, a map $f\colon V(G)\to\Gamma$ is a
$\Gamma$-magic map if the sum of its values over the neighbors of a vertex is
independent of the vertex. It is \emph{affinely generating} if its pairwise
differences generate $\Gamma$. When $|\Gamma|=|V(G)|$, a bijective
$\Gamma$-magic map is a $\Gamma$-distance magic labeling.
For a regular graph $G$ with adjacency matrix $A$, write $\mathbf1$ for the
all-ones vector indexed by $V(G)$, and let $\overline A$ denote the
endomorphism induced by $A$ on
$\Lamone=\mathbb Z^{V(G)}/\mathbb Z\mathbf1$. When $\overline A$ is
nonsingular over $\mathbb Q$, define the reduced Smith group by
$\Sred(G)=\operatorname{coker}\overline A$.
In previous work, we established that for every regular graph $G$ of positive
degree with $\overline A$ nonsingular over $\mathbb Q$, one has
\[
  G\text{ admits an affinely generating }\Gamma\text{-magic map}
  \quad\Longleftrightarrow\quad
  \Gamma\hookrightarrow\Sred(G).
\]
We determine this group explicitly for Kneser graphs. If $r\ge1$, $n\ge2r$, and
$m_j=\binom nj-\binom n{j-1}$, then
\[
  \Sred(K(n,r))\cong
  \bigoplus_{j=1}^{r}
  \left(\mathbb Z/\binom{n-r-j}{r-j}\mathbb Z\right)^{m_j}.
\]
When the labeling group and the Kneser graph have the same order, the
resulting embedding criterion leaves only one case in which an affinely
generating map exists. More precisely, if $\Gamma$ is an abelian group of
order $\binom nr$, then $K(n,r)$ admits an affinely generating
$\Gamma$-magic map if and only if
 $(n,r)=(9,2)$
and
$\Gamma\cong\mathbb Z/6\mathbb Z\oplus\mathbb Z/6\mathbb Z$.
A weak-Sidon-set bound shows that the affinely generating maps in the
exceptional case cannot be bijective.
Hence $K(n,r)$ admits no $\Gamma$-distance magic labeling throughout the
range $r\ge1$ and $n\ge2r$.
\end{abstract}

\maketitle

\section{Introduction}\label{sec:intro}

Let $G$ be a finite simple graph with vertex set $V(G)$, and let $\Gamma$ be a
finite abelian group, written additively. For each $v\in V(G)$, let $N(v)$ denote the
set of neighbors of $v$ in $G$.

\begin{definition}\label{def:gdm}
A map $f\colon V(G)\to\Gamma$ is a \emph{$\Gamma$-magic map} if there exists
an element $\mu\in\Gamma$, called the \emph{magic constant}, such that
\[
  \sum_{u\in N(v)}f(u)=\mu
  \qquad\text{for every }v\in V(G).
\]
The map $f$ is \emph{affinely generating} if
\[
  \left\langle f(v)-f(w)\mid v,w\in V(G)\right\rangle=\Gamma.
\]
When $|\Gamma|=|V(G)|$, a bijective $\Gamma$-magic map is a
\emph{$\Gamma$-distance magic labeling}.
A graph admitting such a labeling is called \emph{$\Gamma$-distance magic}.
\end{definition}

Every $\Gamma$-distance magic labeling is affinely generating. Indeed, one
vertex receives the label $0$, and the differences from that label run through
all of $\Gamma$.

Group-valued distance-magic labelings were introduced in \cite{Fro13,CFSZ16}, whereas the notion of an affinely generating $\Gamma$-magic map was introduced in \cite{Bat26}. For the ordinary integer-valued theory of distance-magic labelings, we refer to the survey \cite{AFK11}.

For a regular graph with adjacency matrix $A$, write $\mathbf1$ for the
all-ones vector indexed by $V(G)$. The sublattice $\mathbb Z\mathbf1$ is
$A$-invariant, so the quotient
\[
  \Lamone:=\mathbb Z^{V(G)}/\mathbb Z\mathbf1
\]
has an induced integral endomorphism $\overline A$. For a
homomorphism $\varphi$, write $\operatorname{Im}\varphi$ for its image. When
$\overline A$ is nonsingular over $\mathbb Q$, define the reduced Smith group
of $G$ by
\[
  \Sred(G):=\operatorname{coker}\overline A
  =\Lamone/\operatorname{Im}\overline A.
\]
For regular graphs of positive degree, Corollary~1.3(b) of \cite{Bat26} gives
the following embedding criterion.
\[
  G\text{ admits an affinely generating }\Gamma\text{-magic map}
  \quad\Longleftrightarrow\quad
  \Gamma\hookrightarrow\Sred(G),
\]
where $\Gamma\hookrightarrow\Sred(G)$ means that there exists an injective
group homomorphism from $\Gamma$ to $\Sred(G)$.
The distance-magic condition also requires bijectivity, which is not captured
by the subgroup-embedding criterion above. The objective here is to determine
the reduced Smith group of a Kneser graph exactly.
Combining the resulting group description with the embedding criterion above,
we show that when $\Gamma$ and a Kneser graph have the same order, an affinely
generating $\Gamma$-magic map exists in only one case. A further obstruction
rules out a $\Gamma$-distance magic labeling in this exceptional case.

\begin{definition}\label{def:kneser}
For integers $r\ge1$ and $n\ge2r$, put $[n]=\{1,\dots,n\}$. The
\emph{Kneser graph} $K(n,r)$ has vertex set $\binom{[n]}r$, with two
$r$-subsets adjacent precisely when they are disjoint.
\end{definition}

The spectrum of $K(n,r)$ consists of the eigenvalues
\begin{equation}\label{eq:kneser-eigenvalues}
  \lambda_j=(-1)^j\binom{n-r-j}{r-j},
  \qquad 0\le j\le r,
\end{equation}
listed with multiplicities
\begin{equation}\label{eq:kneser-multiplicities}
  m_j=\binom nj-\binom n{j-1},
\end{equation}
where $\binom n{-1}=0$; see \cite[Chapter~9]{GR01}. If $n>2r$, the eigenvalues
in \eqref{eq:kneser-eigenvalues} are pairwise distinct. If $n=2r$, repeated
eigenvalues are combined by adding their listed multiplicities. The principal
eigenvalue corresponding to $j=0$ is
\[
  \lambda_0=\binom{n-r}{r}.
\]

The Smith group of the Kneser adjacency matrix is known from Wilson's diagonal
form for inclusion matrices \cite[Theorem~2]{Wil90}. However, it does not by
itself determine the reduced Smith group. Our contribution is the following
explicit determination of the latter.

For a finite abelian group $\Gamma$, let $\ex(\Gamma)$ denote its exponent. For
positive integers $c_1,\dots,c_t$, let
$\operatorname{lcm}_{1\le i\le t}c_i$ denote their least common multiple.
\begin{theorem}\label{thm:kneser-reduced-smith}
Let $r\ge1$ and $n\ge2r$. Then the induced adjacency endomorphism
$\overline A$ on $\Lamone$ is nonsingular over $\mathbb Q$, and
\[
  \Sred(K(n,r))\cong
  \bigoplus_{j=1}^{r}
  \left(\mathbb Z/\binom{n-r-j}{r-j}\mathbb Z\right)^{m_j},
  \qquad
  m_j=\binom nj-\binom n{j-1}.
\]
Consequently,
\[
  \ex\bigl(\Sred(K(n,r))\bigr)
  =\operatorname{lcm}_{1\le j\le r}|\lambda_j|.
\]
\end{theorem}

Note that the summand with $j=r$ is always trivial because
$\binom{n-2r}{0}=1$. To the best of our knowledge, this reduced quotient has not previously
been isolated explicitly.

The exact group in Theorem~\ref{thm:kneser-reduced-smith}, together with the
embedding criterion from \cite[Corollary~1.3(b)]{Bat26}, yields the
graph-theoretic application.

\begin{theorem}\label{thm:kneser-labeling}
Let $r\ge1$ and $n\ge2r$, and let $\Gamma$ be an abelian group of order
$\binom nr$.
\begin{enumerate}[label=\emph{(\alph*)},leftmargin=2.2em]
\item The graph $K(n,r)$ admits an affinely generating $\Gamma$-magic map if
and only if
\[
  (n,r)=(9,2)
  \qquad\text{and}\qquad
  \Gamma\cong\mathbb Z/6\mathbb Z\oplus\mathbb Z/6\mathbb Z.
\]
\item The graph $K(n,r)$ admits no $\Gamma$-distance magic labeling.
\end{enumerate}
\end{theorem}

To prove part~\emph{(a)}, we combine the exact reduced exponent in
Theorem~\ref{thm:kneser-reduced-smith} with prime-divisor results of
Laishram--Shorey and Nagura \cite{LS05,Nag52}. This excludes affine generation
except when
$(n,r)=(9,2)$ and
$\Gamma\cong\mathbb Z/6\mathbb Z\oplus\mathbb Z/6\mathbb Z$. In that case
Theorem~\ref{thm:kneser-reduced-smith} gives
$
  \Sred(K(9,2))\cong(\mathbb Z/6\mathbb Z)^8,
$
so the embedding criterion also proves existence. For part~\emph{(b)}, a
weak-Sidon-set bound of Haanp\"a\"a and \"Osterg\r{a}rd \cite{HO07} shows
that such a map cannot be bijective.

Section~\ref{sec:knesersmith} proves
Theorem~\ref{thm:kneser-reduced-smith}. Section~\ref{sec:tools} collects the
embedding criterion from \cite{Bat26} and the arithmetic and combinatorial
bounds used later. Section~\ref{sec:mainproof}  proves
Theorem~\ref{thm:kneser-labeling}.

\section{The reduced Smith group of a Kneser graph}\label{sec:knesersmith}

Fix integers $r\ge1$ and $n\ge2r$, and write
\[
  V=\binom{[n]}r.
\]
Thus $V$ is the vertex set of $K(n,r)$. Let
$\{e_u\mid u\in V\}$ be the standard basis of the free abelian group
$\mathbb Z^V$. We regard the adjacency matrix $A$ as the matrix of an
endomorphism of $\mathbb Z^V$. For a statement $E$, let $[E]$ denote $1$ when
$E$ is true and $0$ when $E$ is false. Then
\[
  A_{u,v}=[u\cap v=\emptyset],
  \qquad u,v\in V.
\]
We first introduce the standard subsets that will index the coordinates used
in the integral diagonalization of $A$.

For $0\le j\le r$, let $\mathcal B_j$ consist of the sets
\[
  \beta=\{b_1<\cdots<b_j\}\subseteq[n]
\]
such that $b_i\ge2i$ for every $1\le i\le j$. These sets are called
\emph{standard}. The empty set is the unique member of
$\mathcal B_0$. Put
\[
  \mathcal B=\bigcup_{j=0}^r\mathcal B_j.
\]

By \cite{DEGJPP25} (see also \cite[Claim~2.1]{Fra90}),
\begin{equation}\label{eq:level-multiplicities}
  |\mathcal B_j|
  =\binom nj-\binom n{j-1}
  =m_j,
  \qquad 0\le j\le r,
\end{equation}
where the last equality is
\eqref{eq:kneser-multiplicities}. For
$1\le j\le r$, the inequalities
$2j\le2r\le n$ imply
\[
  \frac{\binom nj}{\binom n{j-1}}
  =\frac{n-j+1}{j}>1,
\]
so $m_j>0$. Summing \eqref{eq:level-multiplicities} over $j$ gives
\[
  |\mathcal B|
  =\sum_{j=0}^r m_j
  =\binom nr
  =|V|.
\]

We now fix the matrix-indexing conventions. Choose any total order on $V$ and
use it for both the rows and columns of $A$. For every $j$, choose a total
order on the finite set $\mathcal B_j$. Order $\mathcal B$ by listing the
levels in increasing cardinality,
$\mathcal B_0,\mathcal B_1,\dots,\mathcal B_r$, and use the chosen order
inside each level. Every matrix whose rows and columns are both indexed by
$\mathcal B$ will use this level order on both sides.

Let $\mathbb Z^{\mathcal B}$ be the free abelian group with standard basis
$\{e_\beta\mid\beta\in\mathcal B\}$. Throughout this section, an integral
homomorphism and its matrix with respect to the specified bases are denoted by
the same symbol. Define homomorphisms
\[
  P,Q\colon
  \mathbb Z^{\mathcal B}\longrightarrow\mathbb Z^V
\]
by specifying their values on the standard basis as follows.
\[
  P(e_\beta)
  =\sum_{\substack{u\in V\\\beta\subseteq u}}e_u,
  \qquad
  Q(e_\beta)
  =\sum_{\substack{u\in V\\u\cap\beta=\emptyset}}e_u.
\]
Equivalently, their matrix entries are
\[
  P_{u,\beta}=[\beta\subseteq u],
  \qquad
  Q_{u,\beta}=[u\cap\beta=\emptyset],
\]
where $u\in V$ and $\beta\in\mathcal B$. The fixed order on $V$ indexes the
rows of $P$ and $Q$, and the level order on $\mathcal B$ indexes their columns.
Since $|\mathcal B|=|V|$, both matrices are square. The matrix $P$ is the Bier
matrix and is unimodular, i.e., $\det(P)=\pm1$ \cite{Bie93}.

Both maps identify the empty-set coordinate with the
all-ones vector. Indeed,
every $u\in V$ contains $\emptyset$ and is disjoint from $\emptyset$, so
\begin{equation}\label{eq:emptyset-column}
  P(e_{\emptyset})
  =Q(e_{\emptyset})
  =\sum_{u\in V}e_u
  =\mathbf1.
\end{equation}

For $0\le j\le r$, define
\[
  a_j=\binom{n-r-j}{r-j}.
\]
Since $n\ge2r$, we have $n-r-j\ge r-j$, and hence $a_j$ is a positive integer.
With this notation, \eqref{eq:kneser-eigenvalues} reads
\[
  \lambda_j=(-1)^j a_j.
\]
Define the diagonal endomorphism $D$ of $\mathbb Z^{\mathcal B}$ by
\[
  D(e_\beta)=a_{|\beta|}e_\beta.
\]
Thus its matrix entries are
\[
  D_{\alpha,\beta}=a_{|\beta|}[\alpha=\beta],
  \qquad \alpha,\beta\in\mathcal B.
\]

The following identity can be deduced from \cite[Lemma~1.5]{DEGJPP25} by taking
both the row and column indexing families to consist of the $r$-subsets of
$[n]$ and the prescribed intersection size to be $0$. We include its short
direct proof.

\begin{lemma}[\cite{DEGJPP25}]\label{lem:bier-factor}
The matrices $A$, $P$, $Q$, and $D$ satisfy
$AP=QD.$
\end{lemma}

\begin{proof}
Fix $u\in V$ and $\beta\in\mathcal B_j$, where $0\le j\le r$. From the
definitions of $A$ and $P$,
\[
  (AP)_{u,\beta}
  =
  \sum_{v\in V}
  [u\cap v=\emptyset][\beta\subseteq v].
\]
This sum counts the $r$-subsets $v$ that contain $\beta$ and are disjoint
from $u$. If $u\cap\beta\ne\emptyset$, no such set exists. If
$u\cap\beta=\emptyset$, then every such $v$ is obtained by adjoining to
$\beta$ an $(r-j)$-subset of $[n]\setminus(u\cup\beta)$. This latter set has
$n-r-j$ elements. Consequently,
\[
  (AP)_{u,\beta}
  =
  [u\cap\beta=\emptyset]\binom{n-r-j}{r-j}
  =
  Q_{u,\beta}a_j.
\]
On the other hand, the matrix entries of $Q$ and $D$ give
\[
  (QD)_{u,\beta}
  =
  \sum_{\alpha\in\mathcal B}Q_{u,\alpha}a_j[\alpha=\beta]
  =
  Q_{u,\beta}a_j.
\]
Thus $(AP)_{u,\beta}=(QD)_{u,\beta}$ for every $u\in V$ and
$\beta\in\mathcal B$, which proves $AP=QD$.
\end{proof}

\begin{lemma}\label{lem:bier-mobius}
The matrix $Q$ is unimodular.
\end{lemma}

\begin{proof}
Define an endomorphism $H$ of $\mathbb Z^{\mathcal B}$ by
\[
  H_{\alpha,\beta}
  =(-1)^{|\alpha|}[\alpha\subseteq\beta],
  \qquad \alpha,\beta\in\mathcal B.
\]
We first show that $Q=PH$.
Fix $u\in V$ and $\beta\in\mathcal B$. Expanding the corresponding matrix
entry gives
\begin{equation}\label{eq:ph-entry}
  (PH)_{u,\beta}
  =
  \sum_{\alpha\in\mathcal B}
  [\alpha\subseteq u](-1)^{|\alpha|}[\alpha\subseteq\beta]
  =
  \sum_{\substack{\alpha\in\mathcal B\\
                   \alpha\subseteq u\cap\beta}}
  (-1)^{|\alpha|}
  =
  \sum_{\alpha\subseteq u\cap\beta}
  (-1)^{|\alpha|}.
\end{equation}
The last equality follows from the fact that $\mathcal B$ is closed under
taking subsets. Indeed, write $\beta=\{b_1<\cdots<b_j\}$. If
$\alpha=\{b_{i_1}<\cdots<b_{i_t}\}\subseteq\beta$, then
$b_{i_s}\ge2i_s\ge2s$ for every $1\le s\le t$. Thus $\alpha$ is
standard and hence belongs to $\mathcal B$. The last sum in
\eqref{eq:ph-entry} equals $1$ if $u\cap\beta=\emptyset$. If
$u\cap\beta\ne\emptyset$, then $u\cap\beta$ has equally many subsets of even
and odd cardinality. Indeed, toggling the membership of any fixed element of
$u\cap\beta$ gives a bijection between these two classes of subsets. Hence
the last term of \eqref{eq:ph-entry} is $
  [u\cap\beta=\emptyset]
  =Q_{u,\beta}.
$
Therefore $Q=PH$.

We next show that $H$ is unimodular. If
$H_{\alpha,\beta}\ne0$, then $\alpha\subseteq\beta$ and
$|\alpha|\le|\beta|$. Hence $H$ is block upper triangular with respect to the
level order on $\mathcal B$. If $\alpha$ and $\beta$ have the same
cardinality and $\alpha\subseteq\beta$, then $\alpha=\beta$. The diagonal
block indexed by $\mathcal B_j$ is therefore
\[
  (-1)^jI_{m_j},
\]
where $I_{m_j}$ denotes the $m_j\times m_j$ identity matrix.
It follows that
\[
  \det H=\prod_{j=0}^r(-1)^{j m_j}\in\{1,-1\},
\]
and hence $H$ is unimodular. The identity
$Q=PH$ now shows that $Q$ is unimodular
because both $P$ and $H$ are unimodular.
\end{proof}

Consequently, the column families
\[
  \{P(e_\beta)\mid\beta\in\mathcal B\}
  \quad\text{and}\quad
  \{Q(e_\beta)\mid\beta\in\mathcal B\}
\]
are integral bases of $\mathbb Z^V$. The identity $AP=QD$ is therefore
equivalent to $Q^{-1}AP=D$. Thus the matrix of $A$ with respect to the
$P$-basis in the domain and the $Q$-basis in the codomain is the diagonal
matrix $D$.

\begin{proof}[\textnormal{\textbf{Proof of Theorem~\ref{thm:kneser-reduced-smith}.}}]
The graph $K(n,r)$ is $a_0$-regular. Indeed, if $u\in V$, then every neighbor
of $u$ is an $r$-subset of the $(n-r)$-element set $[n]\setminus u$, and every
such subset is a neighbor. Thus the degree of $u$ is
\[
  \binom{n-r}{r}=a_0.
\]
It follows that $A\mathbf1=a_0\mathbf1$. Hence the constant line
$\mathbb Z\mathbf1$ is $A$-invariant, and $A$ induces an endomorphism
$\overline A$ on
\[
  \Lamone=\mathbb Z^V/\mathbb Z\mathbf1.
\]
Explicitly,
\[
  \overline A(x+\mathbb Z\mathbf1)
  =Ax+\mathbb Z\mathbf1,
  \qquad x\in\mathbb Z^V.
\]

Since $P$ is unimodular by \cite{Bie93} and $Q$ is unimodular by
Lemma~\ref{lem:bier-mobius}, both homomorphisms are isomorphisms.
Equation~\eqref{eq:emptyset-column} gives
\[
  P(\mathbb Ze_{\emptyset})
  =Q(\mathbb Ze_{\emptyset})
  =\mathbb Z\mathbf1.
\]
Therefore, $P$ and $Q$ induce isomorphisms
\[
  \overline P,\overline Q\colon
  \mathbb Z^{\mathcal B}/\mathbb Ze_{\emptyset}
  \longrightarrow
  \Lamone.
\]
Explicitly,
\[
  \overline P(x+\mathbb Ze_{\emptyset})
  =Px+\mathbb Z\mathbf1,
  \qquad
  \overline Q(x+\mathbb Ze_{\emptyset})
  =Qx+\mathbb Z\mathbf1.
\]

The equality
\[
  D(e_{\emptyset})=a_0e_{\emptyset}
\]
shows that $D$ preserves $\mathbb Ze_{\emptyset}$. It therefore
induces an endomorphism $\overline D$ of
$\mathbb Z^{\mathcal B}/\mathbb Ze_{\emptyset}$ given by
\[
  \overline D(x+\mathbb Ze_{\emptyset})
  =Dx+\mathbb Ze_{\emptyset}.
\]
The identity $AP=QD$ from Lemma~\ref{lem:bier-factor} descends to the quotient
lattices and gives
\[
  \overline A\,\overline P
  =\overline Q\,\overline D.
\]
Since $\overline P$ is an isomorphism,
\begin{equation}\label{eq:reduced-factorization}
  \overline A
  =\overline Q\,\overline D\,\overline P^{-1},
\end{equation}
which shows that $\overline A$ and $\overline D$ are integrally
equivalent.

Every class in $\mathbb Z^{\mathcal B}/\mathbb Ze_{\emptyset}$ has a unique
representative whose $e_{\emptyset}$-coordinate is zero. Hence
\[
  \mathbb Z^{\mathcal B}/\mathbb Ze_{\emptyset}
  \cong
  \bigoplus_{\beta\in\mathcal B\setminus\{\emptyset\}}
  \mathbb Ze_\beta
  =
  \bigoplus_{j=1}^r\mathbb Z^{\mathcal B_j}.
\]
With respect to this decomposition,
\begin{equation}\label{eq:reduced-diagonal}
  \overline D=\bigoplus_{j=1}^r a_jI_{m_j}.
\end{equation}
Since every $a_j$ is positive,
$\overline D$ is nonsingular over $\mathbb Q$. Integral equivalence shows
that $\overline A$ is also nonsingular over $\mathbb Q$, and hence
$\Sred(K(n,r))$ is defined.

By \eqref{eq:reduced-factorization} and the surjectivity of
$\overline P^{-1}$,
\[
  \operatorname{Im}\overline A
  =\overline Q\bigl(\operatorname{Im}\overline D\bigr).
\]
Therefore $\overline Q$ induces an isomorphism
\[
  \operatorname{coker}\overline D
  \longrightarrow
  \operatorname{coker}\overline A,
  \qquad
  x+\operatorname{Im}\overline D
  \longmapsto
  \overline Q(x)+\operatorname{Im}\overline A.
\]
By \eqref{eq:reduced-diagonal}, the restriction of $\overline D$ to
$\mathbb Z^{\mathcal B_j}$ is multiplication by $a_j$. Consequently,
\[
  \Sred(K(n,r))
  =\operatorname{coker}\overline A
  \cong
  \operatorname{coker}\overline D
  \cong
  \bigoplus_{j=1}^r
  \left(\mathbb Z/a_j\mathbb Z\right)^{m_j}.
\]
Substituting $a_j=\binom{n-r-j}{r-j}$ gives the decomposition in the theorem.

The inequality following \eqref{eq:level-multiplicities} shows that $m_j>0$ for
every $1\le j\le r$. Thus every value $a_j$ in this range occurs as a
cyclic-summand order. The exponent of a finite direct sum of cyclic groups is
the least common multiple of the orders of its cyclic summands.
Since $|\lambda_j|=a_j$, we conclude that
\[
  \ex\bigl(\Sred(K(n,r))\bigr)
  =\operatorname{lcm}_{1\le j\le r}a_j
  =\operatorname{lcm}_{1\le j\le r}|\lambda_j|.
\]
\end{proof}

\section{Auxiliary criteria and bounds}\label{sec:tools}

All groups considered are finite, abelian, and written additively.

\begin{lemma}\label{lem:radical}
A prime $p$ divides $|\Gamma|$ if and only if $p$ divides $\ex(\Gamma)$.
\end{lemma}

\begin{proof}
Every element order divides $|\Gamma|$, so $\ex(\Gamma)\mid|\Gamma|$.
Conversely, if $p\mid|\Gamma|$, Cauchy's theorem gives an element of order
$p$.
\end{proof}

\subsection{The embedding criterion}

The following result is Corollary~1.3(b)--(c) of \cite{Bat26}.

\begin{theorem}[\cite{Bat26}]
\label{thm:reduced-criterion}
Let $G$ be a regular graph of positive degree, let $\Gamma$ be a finite abelian
group, and assume that the induced adjacency endomorphism $\overline A$ on
$\Lamone$ is nonsingular over $\mathbb Q$. Then $G$ admits an affinely
generating $\Gamma$-magic map if and only if
\[
  \Gamma\hookrightarrow\Sred(G).
\]
Whenever these equivalent conditions hold,
\[
  \ex(\Gamma)\mid\ex(\Sred(G)).
\]
Moreover, if $|\Gamma|=|V(G)|$ and $G$ admits a
$\Gamma$-distance magic labeling, then that labeling is affinely generating,
and hence $\Gamma\hookrightarrow\Sred(G)$.
\end{theorem}

\begin{corollary}\label{cor:prime-obstruction}
Under the hypotheses of Theorem~\ref{thm:reduced-criterion}, assume in addition
that $|\Gamma|=|V(G)|$. If $G$ admits an affinely generating $\Gamma$-magic
map, then every prime divisor of $|V(G)|$ divides $\ex(\Sred(G))$.
\end{corollary}

\begin{proof}
If $p\mid|V(G)|=|\Gamma|$, then $p\mid\ex(\Gamma)$ by
Lemma~\ref{lem:radical}. Apply Theorem~\ref{thm:reduced-criterion}.
\end{proof}

\subsection{Prime-divisor bounds}
For an integer $m>1$, let $P^{+}(m)$ denote its greatest prime factor. We use
the following two explicit results.

\begin{theorem}[\cite{Nag52}]\label{thm:nagura}
For every real number $x\ge25$, there exists a prime $p$ satisfying
$x<p<1.2x$.
\end{theorem}

\begin{theorem}[\cite{LS05}]\label{thm:LS}
Let $k\ge2$ and let $x$ be an integer satisfying
\[
 x\ge\max\left\{k+13,\frac{279}{262}k\right\}.
\]
Then
\[
  P^{+}\bigl((x+1)(x+2)\cdots(x+k)\bigr)>2k.
\]
\end{theorem}

\subsection{Weak Sidon sets}
A subset $S$ of an abelian group $\Gamma$ is a \emph{weak Sidon set} if
\[
  a+b=c+d,\qquad a\ne b,\quad c\ne d,
\]
with $a,b,c,d\in S$, implies $\{a,b\}=\{c,d\}$. Equivalently, the sums indexed
by the $2$-element subsets of $S$ are pairwise distinct. This notion was
studied by Ruzsa \cite{Ruz93} and appeared earlier under the name of well-spread
sets \cite{Kot72}. Let
$\mathrm{Ord}(\Gamma,2)$ denote the set of elements of order two in $\Gamma$.
The following formulation of the Haanp\"a\"a--\"Osterg\r{a}rd bound is taken
from Bajnok's restatement \cite{Baj18}.

\begin{theorem}[\cite{HO07}]\label{thm:HO}
Let $\Gamma$ be an abelian group, and let $S\subseteq\Gamma$ be a
weak Sidon set. Then
\[
  |S|\le
  \left\lfloor
    \frac{\sqrt{4|\Gamma|+4|\mathrm{Ord}(\Gamma,2)|+5}+3}{2}
  \right\rfloor.
\]
\end{theorem}

\section{Affinely generating magic maps and distance magic labelings}
\label{sec:mainproof}

We first combine Theorem~\ref{thm:kneser-reduced-smith} with the embedding
criterion.

\begin{corollary}\label{cor:kneser-affine}
Let $r\ge1$ and $n\ge2r$, and let $\Gamma$ be a finite abelian group. Then
$K(n,r)$ admits an affinely generating $\Gamma$-magic map if and only if
\[
  \Gamma\hookrightarrow
  \bigoplus_{j=1}^{r}
  \left(\mathbb Z/\binom{n-r-j}{r-j}\mathbb Z\right)^{m_j},
  \qquad
  m_j=\binom nj-\binom n{j-1}.
\]
In particular, affine generation implies
\[
  \ex(\Gamma)\mid
  L_{n,r}:=\operatorname{lcm}_{1\le j\le r}
  \binom{n-r-j}{r-j}.
\]
\end{corollary}

\begin{proof}
The equivalence follows from Theorem~\ref{thm:kneser-reduced-smith} and
Theorem~\ref{thm:reduced-criterion}. The exponent divisibility follows from
the subgroup embedding.
\end{proof}

For groups of order $\binom nr$, Corollary~\ref{cor:prime-obstruction} and
Theorem~\ref{thm:kneser-reduced-smith} show that affine generation forces
every prime divisor of $\binom nr$ to divide $L_{n,r}$. We use the following
two consequences of this condition.

\begin{lemma}\label{lem:primediv}
Let $r\ge2$ and $n>2r$. Assume every prime divisor of $\binom{n}{r}$ divides
$L_{n,r}$. Then
\begin{enumerate}[label=\emph{(\alph*)},nosep,leftmargin=2.2em]
\item the interval $(n-r,n]$ contains no prime;
\item
\[
 P^{+}\bigl(n(n-1)\cdots(n-r+1)\bigr)<2r.
\]
\end{enumerate}
\end{lemma}

\begin{proof}
(a) Suppose a prime $p$ satisfies $n-r<p\le n$. Since $n>2r$, we have
$p>n-r>r$. Thus $p$ divides neither $r!$ nor $(n-r)!$, whereas $p\mid n!$,
so $p\mid\binom{n}{r}$. By hypothesis, $p\mid L_{n,r}$, and hence
$p\mid\binom{n-r-j}{r-j}$ for some $1\le j\le r$. But
$n-r-j<p$, so none of the factors in the numerator of this binomial coefficient
is divisible by $p$, a contradiction.

(b) Let $p$ be a prime divisor of $n(n-1)\cdots(n-r+1)$, say $p\mid n-a$ with
$0\le a\le r-1$. If $p\le r$, then $p<2r$. Suppose instead that $p>r$.
Then $p\nmid r!$, so $p\mid\binom{n}{r}$ and hence $p\mid L_{n,r}$. Thus
$p\mid\binom{n-r-j}{r-j}$ for some $1\le j\le r-1$. Since
$p>r\ge r-j$, the denominator $(r-j)!$ is coprime to $p$, so $p$ divides one
of the numerator factors
\[
  n-r-j,\ n-r-j-1,\dots,\ n-2r+1.
\]
Hence $p\mid n-b$ for some $b$ with $r+j\le b\le2r-1$. In particular,
$r+1\le b\le2r-1$. Therefore $p\mid b-a$, and since
$0<b-a\le2r-1$, we obtain $p\le2r-1$.
\end{proof}

\begin{proposition}\label{prop:arith}
Let $\Gamma$ be an abelian group of order $\binom{n}{r}$. If $K(n,r)$ admits an
affinely generating $\Gamma$-magic map, then
\[
  (n,r)=(9,2)
  \qquad\text{and}\qquad
  \Gamma\cong\mathbb Z/6\mathbb Z\oplus\mathbb Z/6\mathbb Z.
\]
\end{proposition}

\begin{proof}
Corollary~\ref{cor:kneser-affine} first excludes the boundary cases. If $r=1$,
then $L_{n,1}=1$. If $n=2r$, then every binomial coefficient in $L_{n,r}$ is
$1$. In either case affine generation would force $\ex(\Gamma)=1$, which is
impossible because $|\Gamma|=\binom nr>1$. Hence $r\ge2$ and $n>2r$.

By Corollary~\ref{cor:prime-obstruction} and
Theorem~\ref{thm:kneser-reduced-smith}, every prime divisor of
$\binom nr$ divides $L_{n,r}$. Thus Lemma~\ref{lem:primediv} applies. We use
Theorem~\ref{thm:LS} with $k=r$ and $x=n-r$. If
\[
  n-r\ge\max\left\{r+13,\frac{279}{262}r\right\},
\]
then
\[
  P^{+}\bigl(n(n-1)\cdots(n-r+1)\bigr)>2r,
\]
contrary to Lemma~\ref{lem:primediv}(b). Hence
\begin{equation}\label{eq:nbound}
  n<\max\left\{2r+13,\frac{541}{262}r\right\}.
\end{equation}

Suppose first that $r\ge25$. Both quantities on the right of
\eqref{eq:nbound} are smaller than $6r$, so $n<6r$ and hence
$1.2(n-r)<n$. Put $x=n-r$. Since $x>r\ge25$,
Theorem~\ref{thm:nagura} gives a prime $p$ with
\[
  n-r=x<p<1.2x<n,
\]
contradicting Lemma~\ref{lem:primediv}(a).

It remains to consider $2\le r\le24$. In this range
$\frac{541}{262}r\le2r+13$, so \eqref{eq:nbound} gives
$n\le2r+12\le60$. By Lemma~\ref{lem:primediv}(a), the $r$ consecutive integers
\[
  n-r+1,n-r+2,\dots,n
\]
are all composite. A direct check of the integers up to $60$ shows that the
maximum length of a block of consecutive composites is $5$, and that all
blocks of this length are
\[
  24,25,26,27,28;\qquad
  32,33,34,35,36;\qquad
  48,49,50,51,52;\qquad
  54,55,56,57,58.
\]
Thus $r\le5$. For $r\in\{4,5\}$ we have $n\le22$, and there is no run of four
consecutive composites up to $22$. Hence only $r=2,3$ remain.

For $r=2$, the values $n\le16$ for which $n-1$ and $n$ are both composite are
$9,10,15,16$. For $r=3$, the values $n\le18$ for which $n-2,n-1,n$ are all
composite are $10,16$. We now test the prime-divisor obstruction from
Corollary~\ref{cor:prime-obstruction}, using
Theorem~\ref{thm:kneser-reduced-smith}:
\[
\begin{array}{c|c|c|c|c}
(n,r)&\binom nr&L_{n,r}&\text{prime divisors of }\binom nr
 &\text{all divide }L_{n,r}\\ \hline
(9,2)  &36 &6  &2,3   &\text{yes}\\
(10,2) &45 &7  &3,5   &\text{no}\\
(15,2) &105&12 &3,5,7 &\text{no}\\
(16,2) &120&13 &2,3,5 &\text{no}\\
(10,3) &120&15 &2,3,5 &\text{no}\\
(16,3) &560&66 &2,5,7 &\text{no}
\end{array}
\]
Thus $(n,r)=(9,2)$ is the only possible pair.

In this case $|\Gamma|=36$ and $L_{9,2}=6$. Corollary~\ref{cor:kneser-affine}
gives $\ex(\Gamma)\mid6$, while Lemma~\ref{lem:radical} shows that both $2$ and
$3$ divide $\ex(\Gamma)$. Hence $\ex(\Gamma)=6$. The unique abelian group of
order $36$ and exponent $6$ is
\[
  \Gamma\cong\mathbb Z/6\mathbb Z\oplus\mathbb Z/6\mathbb Z.\qedhere
\]
\end{proof}

\begin{proof}[\textnormal{\textbf{Proof of Theorem~\ref{thm:kneser-labeling}.}}]
\emph{(a)} The forward implication is Proposition~\ref{prop:arith}. Conversely, suppose
that $(n,r)=(9,2)$ and
$\Gamma\cong\mathbb Z/6\mathbb Z\oplus\mathbb Z/6\mathbb Z$.
Theorem~\ref{thm:kneser-reduced-smith} gives
\[
  \Sred(K(9,2))
  \cong
  (\mathbb Z/6\mathbb Z)^8.
\]
Thus $\Gamma\hookrightarrow\Sred(K(9,2))$, and
Corollary~\ref{cor:kneser-affine} gives an affinely generating
$\Gamma$-magic map.

\medskip
\noindent\emph{(b)} Assume that $K(n,r)$ admits a $\Gamma$-distance magic labeling $\ell$, where
$|\Gamma|=\binom{n}{r}$. The labeling $\ell$ is affinely generating, so
part~\emph{(a)} gives
\[
  (n,r)=(9,2),\qquad
  \Gamma\cong\mathbb Z/6\mathbb Z\oplus\mathbb Z/6\mathbb Z.
\]
For each $i\in[9]$, put
\[
  T_i=\sum_{\substack{u\in\binom{[9]}2\\i\in u}}\ell(u),
\]
and put
\[
  \sigma=\sum_{u\in\binom{[9]}2}\ell(u).
\]
Let $\mu$ be the magic constant. If $\{i,j\}$ is a vertex, then its
neighbors are precisely the pairs containing neither $i$ nor $j$. Therefore
\[
  \mu=\sigma-T_i-T_j+\ell(\{i,j\}),
\]
and hence
\[
  \ell(\{i,j\})=\mu-\sigma+T_i+T_j.
\]
Since $\ell$ is bijective, the $36$ sums
\[
  T_i+T_j,\qquad 1\le i<j\le9,
\]
are pairwise distinct. If $T_i=T_j$ for some $i\ne j$, then choosing
$h\notin\{i,j\}$ would give $T_i+T_h=T_j+T_h$, a contradiction. Hence
$T_1,\dots,T_9$ are pairwise distinct, and
\[
  \mathcal T=\{T_1,\dots,T_9\}
\]
is a weak Sidon set of size $9$ in $\Gamma$.

The group $\mathbb Z/6\mathbb Z\oplus\mathbb Z/6\mathbb Z$ has exactly three
elements of order two. Therefore
Theorem~\ref{thm:HO} gives
\[
  |\mathcal T|
  \le
  \left\lfloor
    \frac{\sqrt{4\cdot36+4\cdot3+5}+3}{2}
  \right\rfloor
  =\left\lfloor\frac{\sqrt{161}+3}{2}\right\rfloor
  =7,
\]
contradicting $|\mathcal T|=9$.
\end{proof}


\end{document}